\documentclass{amsart}
\usepackage[all]{xy}
\usepackage{amsmath,amssymb,amsthm,mathrsfs,bm}

\theoremstyle{definition}
\newtheorem{dfn}{Definition}[section]

\theoremstyle{plain}
\newtheorem{prop}[dfn]{Proposition}
\newtheorem{lem}[dfn]{Lemma}
\newtheorem{thm}[dfn]{Theorem}

\def \O{\mathcal{O}}
\def \Kab{K^{\mathrm{ab}}}
\def \GKab{G_K^{\mathrm{ab}}}

\def \p{\mathfrak{p}}

\def \a{\mathfrak{a}}
\def \b{\mathfrak{b}}

\def \Af{\mathbb{A}_{K,f}}
\def \R{\mathbb{R}}

\def \Q{\mathbb{Q}}
\def \C{\mathbb{C}}
\def \Z{\mathbb{Z}}
\def \N{\mathbb{N}}
\def \BB{\mathbb{B}}
\def \KK{\mathbb{K}}

\def \Prim{\mathrm{Prim}}
\def \Ind{\mathrm{Ind}}

\title{Bost-Connes systems for local fields of characteristic zero}
\author{Takuya Takeishi}
\address{Department of Mathematical Sciences, University of Tokyo}
\email{takeishi@ms.u-tokyo.ac.jp}
\date{}

\begin{document}

\begin{abstract}
The Bost-Connes system is a $C^*$-dynamical system which has a relation with the class field theory. 
The purpose of this paper is to generalize the notion of Bost-Connes systems to local fields of characteristic zero. 
The notable phenomenon in the case of local fields is the absence of the phase transition. 
We also investigate the relation between the global and local ones. 
\end{abstract}

\maketitle

\section{Introduction}

The Bost-Connes system is a $C^*$-dynamical system which describes the relation
between quantum statistical mechanics and the class field theory. This system is first
constructed by Bost and Connes in 1995 (\cite{BC}) for $\Q$. 
After the work of Bost-Connes, there were a lot of attempts to generalize such a system to other fields. 
The notable one is the work of Connes-Marcolli-Ramachandran, which succeeded in generalizing such system for imaginary quadratic fields 
in a complete way in 2005 (\cite{CMR}). 
Also, the requirement of Bost-Connes systems is explicitly axiomatized in that paper. 
For a number field $K$, $\Kab$ denotes the maximal abelian extension of $K$ and $\GKab=G(\Kab/K)$ denotes its Galois group over $K$. 
\begin{dfn}[cf. \cite{CMR}]
For a number field $K$, 
a $C^*$-dynamical system $(A, \sigma_t)$ equipped with an action of $\GKab$ is called a Bost-Connes type system for $K$ 
if it satisfies the following axioms:
\begin{enumerate}
\item There exists a unique KMS${_\beta}$-state of $(A, \sigma_t)$ for $0<\beta \leq 1$. 
\item For $1 < \beta \leq +\infty$, the action of $\GKab$ on the set of extremal KMS$_{\beta}$-states of $(A, \sigma_t)$ is free and transitive. 
\item The partition function of $(A, \sigma_t)$ is $\zeta_K(\beta)$. 
\item There exists a $K$-subalgebra $A_0\subset A$ satisfying the following conditions ($A_0$ is called an {\it arithmetic subalgebra}): 
\begin{enumerate}
\item The subalgebra $A_0 \otimes_K \C$ is dense in $A$. 
\item For any extremal KMS$_{\infty}$-states $\varphi$, we have $\varphi(A_0) = \Kab$. 
\item For any extremal KMS$_{\infty}$-states $\varphi$, $\alpha \in \GKab$ and $a \in A_0$, we have $\varphi(\alpha(a)) = \alpha (\varphi(a))$. 
\end{enumerate}
\end{enumerate}
\end{dfn}

Although the existence of Bost-Connes system for arbitrary $K$ was a longstanding problem, it was completely solved by Ha-Paugam \cite{HP}, 
Laca-Larsen-Neshveyev \cite{LLN} and Yalkinoglu \cite{Y}. 
Ha-Paugam's work is devoted for the explicit construction of $C^*$-dynamical systems. 
Laca-Larsen-Neshveyev classified KMS-states of Ha-Paugam's ones. 
Finally, Yalkinoglu proved the existence of an arithmetic subalgebra. 
In this paper, the terminology of {\it the Bost-Connes systems} indicates the specific $C^*$-dynamical systems constructed by them, 
and it is distinguished by the terminology of {\it the Bost-Connes type system} (for an explicit construction, \cite{Y} is a good reference). 
Note that this distinction is not usual, and we use a similar distinction for the case of local fields. 

The purpose of this paper is to extend this notion to local fields of characteristic zero, and to observe a basic relation between global and local ones. 
For local fields, we adopt almost the same axioms, but we have to change it slightly in order to fit the local case. 
For a local field $K$, the module of $K$ is denoted by $q_K$. 
If $K$ is a finite extension of $\Q_{p}$, then $q_K= p^f$ where $f$ is the inertia degree of $K/\Q_p$. 
For local fields, the appropriate axioms are the following: 
\begin{dfn} \label{loc}
Let $K$ be a local field of characteristic zero. 
Fix an embedding $K \subset \C$ and regard $\Kab$ as a subfield of $\C$. 
A $C^*$-dynamical system $(A, \sigma_t)$ equipped with an action of $\GKab$ is called a Bost-Connes type system for $K$ 
if it satisfies the following axioms:
\begin{enumerate}
\item For $0 < \beta \leq +\infty$, the action of $\GKab$ on the set of extremal KMS$_{\beta}$-states of $(A, \sigma_t)$ is free and transitive. 
\item The partition function of $(A, \sigma_t)$ is $(1-q_K^{-\beta})^{-1}$. 
\item There exists a $K$-subalgebra $A_0\subset A$ satisfying the following conditions ($A_0$ is called an {\it arithmetic subalgebra}): 
\begin{enumerate}
\item The subalgebra $A_0 \otimes_K \C$ is dense in $A$. 
\item For any extremal KMS$_{\infty}$-states $\varphi$, we have $\varphi(A_0) = \Kab$. 
\item For any extremal KMS$_{\infty}$-states $\varphi$, $\alpha \in \GKab$ and $a \in A_0$, we have $\varphi(\alpha(a)) = \alpha (\varphi(a))$. 
\end{enumerate}
\end{enumerate}
\end{dfn}

Compared to the case of number fields, we do not require the axiom of the phase transition, because it conflicts with the condition (2). 
In order to define the notion of an arithmetic subalgebra, we have to fix an unnatural embedding $K\subset \C$. 

The main theorem of this paper is the following: 
\begin{thm} \label{main} 
For any local field $K$ of characteristic zero, there exists a Bost-Connes type system for $K$. 
\end{thm}

In Section \ref{construction}, we construct a $C^*$-dynamical system and prove that it satisfies the above axioms. 
We construct an $\R$-equivariant $C^*$-correspondence between Bost-Connes systems for number fields and its localizations in Section \ref{connection}, 
and investigate the behaviour of KMS-states. 
The $C^*$-algebraic structure of Bost-Connes systems for local fields is rather easy compared with that of number fields. 
Fundamental invariants, such as $K$-theory and the primitive ideal space, is observed in Section \ref{feature}.

\section{Construction of Bost-Connes system for local fields} \label{construction}

\subsection{The Definition}

In this section, we would like to give an explicit construction of Bost-Connes systems for local fields. 
The construction is similar to the case of number fields, using the local class field theory. 

Let $K$ be a local field of characteristic zero. 
The integer ring of $K$ is denoted by $\O_K$ and the module of $K$ is denoted by $q_K$.  
Set $Y_K = \O_K \times_{\O^*_K} \GKab$. 
Note that $Y_K$ is compact and Hausdorff. 
We can define an action of $\N$ by 
\begin{equation*}
k\cdot [\rho,\alpha] = [\rho \pi^k,[\pi]_K^{-k}\alpha],
\end{equation*}
where $\pi$ is a prime element of $K$, and $[\cdot]_K$ denotes the reciprocity map of $K$. 
This action is obviously well-defined and independent of the choice of a prime element. 

\begin{dfn}
The Bost-Connes $C^*$-algebra for $K$ is the following semigroup crossed product, 
\begin{equation*}
A_K = C(Y_K) \rtimes \N.
\end{equation*}
\end{dfn}

The time evolution on $A_K$ is defined by 
\begin{equation*}
\sigma_t(fv_k) = q_K^{itk}fv_k. 
\end{equation*}

\begin{dfn}
The dynamical system 
$\mathcal{A}_K = (A_K, \sigma_t)$
is called the Bost-Connes system for $K$. 
\end{dfn}

As in the case of number fields, $A_K$ can be written as a full-corner of a group crossed products. 
Let 
\begin{equation*}
X_{K} = K \times_{\O_K^*} \GKab. 
\end{equation*}
Then we can define the action of $\Z$ on $X_K$ in the same way. 
Let $\tilde{A}_K = C_0(X_K) \rtimes \Z$. Then we have $A_K = 1_{Y_K} \tilde{A}_K 1_{Y_K}$ and $1_{Y_K}$ is a full projection. 

Compared to the case of number fields, the $C^*$-algebraic structure of $A_K$ is quite easy, 
so the additional information (i.e., the time evolution) has a crucial role. 
\begin{prop} \label{identical}
If $K$ and $L$ are two local fields of characteristic zero, then $A_K$ is isomorphic to $A_L$ as $C^*$-algebras. 
\end{prop}
\begin{proof}
By looking at the valuation, we have the decomposition $K = (\Z \times \O_K^*) \cup \{0\}$ and $\O_K = (\N \times \O_K^*) \cup \{0\}$. 
By the local class field theory, $\GKab$ can be identified with $\hat{\Z} \times \O_K^*$, where $\hat{\Z}$ is the profinite completion of $\Z$. 
So the non-zero part is identified with $\N \times \hat{\Z} \times \O_K^*$ and zero-part is identified with $\hat{\Z}$. 
Under this decomposition, the action of $\N$ is described as follows: 
\begin{eqnarray*}
&&n(k,x,b) = (k+n,x-n,b) \mbox{ for } n \in \Z, k \in \N, x \in \hat{\Z}, b \in \O_K^* \mbox{ on } \N \times \hat{\Z} \times \O_K^*, \\ 
&&nx = x-n \mbox{ for } x \in \hat{\Z} \mbox{ on } \hat{\Z}. 
\end{eqnarray*}
The space $\O_K^*$ is just a Cantor set as a topological space. Hence the dynamical system $(Y_K,\N)$ is conjugate to $(Y_L,\N)$. 
\end{proof}

For $C^*$-dynamical systems, $(A_K,\sigma_t)$ is completely classified with the module $q_K$. 

We would like to prepare several notations. 
Let 
\begin{eqnarray*}
Y_K^* &=& \O_K^* \times_{\O_K^*} \GKab, \\
Y_K^0 &=& \{0\} \times_{\O_K^*} \GKab, \\
Y_K^{\natural} &=& (\O_K \setminus \{0\}) \times_{\O_K^*} \GKab, \mbox{ and} \\ 
X_K^{\natural} &=& K^* \times_{\O_K^*} \GKab. 
\end{eqnarray*}
\subsection{Classification of KMS-states}

We classify KMS-states of $\mathcal{A}_K = (A_K, \sigma_t)$ in this section. 
Note that we distinguish the terminology of KMS$_{\infty}$-states and ground states. 
A ground state is said to be a KMS$_{\infty}$-state if it is written as a limit of KMS$_{\beta}$-states for $\beta \rightarrow \infty$ 
(for the precise definition, see \cite[p.408]{CMb}). 
However, it does not matter in our case because all ground states of the Bost-Connes system are KMS$_{\infty}$-states. 

The proof follows the line of Laca-Larsen-Neshveyev \cite{LLN}. 
We use the groupoid presentation of Bost-Connes $C^*$-algebras (although this may not be crucial --- we use it just to apply Propositions in \cite{LLN}). 
Define a groupoid $\mathcal{G}_K$ by 
\begin{equation*}
\mathcal{G}_K = \{ (k,x) \in \Z \ltimes X_K\ |\ x \in Y_K,\ kx \in Y_K\}. 
\end{equation*} 
Then $C^*_r(\mathcal{G}_K)$ is isomorphic to $A_K$ as $C^*$-algebras. 
Note that $\mathcal{G}_K$ is the restriction of the transformation groupoid $\Z \ltimes X_K$ onto $Y_K$. 
The time evolution is determined by 
\begin{equation*}
\sigma_t(f)(k,x) = q_K^{itk}f(k,x)
\end{equation*}
for $f \in C_c(\mathcal{G}_K)$. 
The conditional expectation $C^*_r(\mathcal{G}_K) \rightarrow C(Y_K)$ is also denoted by $E$. 

We immediately obtain the following lemma by \cite[Proposition 1.1]{LLN}: 
\begin{lem} \label{lem1}
For $0< \beta < \infty$, the set of KMS$_{\beta}$-states on $C^*_r(\mathcal{G}_K)$ is affine isomorphic to the set of 
Radon measures $\mu$ on $X_K$ with $\mu (Y_K) =1$ and satisfying 
\begin{equation*}
\mu(kZ) = q_K^{-\beta k}\mu(Z) 
\end{equation*}
for Borel subsets $Z$. The measure $\mu$ corresponds to the restriction of the state $\mu \circ E$ on $C^*_r(\mathcal{G}_K)$. 
\end{lem}

From now on, measures are tacitly assumed to be Radon measures. 
For a measure $\mu$ which satisfies the condition in Lemma \ref{lem1}, we have $\mu(Y_K^0)=0$ by the scaling formula. 
Hence $\mu$ is concentrated on $X_K^{\natural}$, and by using the scaling formula again, $\mu$ is completely determined by $\nu=\mu|_{Y_K^*}$. 

\begin{lem}{\rm (cf. \cite[Proposition 1.2]{LLN})} \label{lem2}
The set of KMS$_{\beta}$-states on $C^*_r(\mathcal{G}_K)$ is affine isomorphic to the set of probability measures on $Y_K^*$. 
The probability measure $\nu$ on $Y_K^*$ corresponds to the state $\varphi = \mu \circ E$, where $\mu$ is the measure of $X_K$ satisfying
$\mu|_{Y_K^*}=\nu$ and 
\begin{equation*}
\mu (Z) = (1-q_K^{-\beta})\sum_{k=-\infty}^{\infty} q_K^{-k\beta}\nu ((-k)Z \cap Y_K^*) 
\end{equation*}
for Borel subsets $Z$ of $X_K$. 
\end{lem}

Taking the limit of $\beta \rightarrow \infty$, we can see that the summands of $k \neq 0$ disappear. 
This means that $\varphi = \mu \circ E$ is a KMS$_{\infty}$-state for any probability measure $\mu$ on $Y_K^*$. 
In order to see that all ground states are of this form, we can apply \cite[Proposition 1.3]{LLN} for $Y_0 = Y_K^*$. 
We obtain the following lemma: 
\begin{lem} \label{lem3}
The set of ground states of $C^*_r(\mathcal{G}_K)$ is affine isomorphic to 
the set of probability measures on $Y_K^*$. 
The probability measure $\mu$ on $Y_K^*$ corresponds to the state $\varphi = \mu \circ E$. 
\end{lem}

In summary, we obtain the following proposition: 
\begin{prop} \label{KMS}
We have the following : 
\begin{enumerate} 
\item For $0<\beta<\infty$, we have a one-to-one correspondence between KMS$_{\beta}$-states on $\mathcal{A}_K$ 
and probability measures on $\GKab$. 
The Dirac measure on $w \in \GKab$ corresponds to the extremal KMS$_{\beta}$-state $\varphi_{\beta,w}$ given by 
\begin{equation*}
\varphi_{\beta,w}(f) = (1-q_K^{-\beta})\sum_{k=0}^{\infty} q_K^{-k\beta} f(k \cdot w)  
\end{equation*}
for $f \in C(Y_K)$. 
\item We have a one-to-one correspondence between KMS$_{\infty}$ on $\mathcal{A}_K$ and probability measures on $\GKab$. 
The dirac measure on $w \in G(\Kab/K)$ corresponds the extremal KMS$_{\infty}$-state $\varphi_{\infty,w}$ given by 
\begin{equation*}
\varphi_{\beta,w}(f) = f([1,w]) 
\end{equation*}
for $f \in C(Y_K)$. 
\item For $0<\beta < \infty$, all extremal KMS$_{\beta}$-states are of type I and the partition function is $(1-q_K^{-\beta})^{-1}$. 
\end{enumerate}
\end{prop}

\begin{proof}
We have already seen (1)--(3) in Proposition \ref{KMS}, 
and the convergence of $\varphi_{\beta, \mu}$ to $\varphi_{\infty, \mu}$ in norm for every probability measure $\mu$ of $\GKab$. 
So it suffices to prove (4). 

In order to prove that $\varphi_{\beta, w}$ is of type I, 
it suffices to find an irreducible representation $\pi_w:A_K \rightarrow B(\mathcal{H}_w)$ and a self-adjoint operator $H_w$ on $\mathcal{H}_w$ 
such that $e^{itH_w} \pi_w(a) e^{-itH_w} = \pi_w(\sigma_t(a))$ and 
\begin{equation*}
\varphi_{\beta, w} (a) = \frac{\mathrm{Tr}(e^{-\beta H_w}\pi_w(a))}{\mathrm{Tr}(e^{-\beta H_w})} 
\end{equation*}
for $a \in A_K$. 
Define $\pi_w : A_K \rightarrow B(\ell^2\N)$ by 
\begin{eqnarray*}
\pi_w (f) \xi_k &=& f(kw) \xi_k \mbox{ for } f \in C(Y_K) \mbox{ and} \\
\pi_w (v_l) \xi_k &=& \xi_{k+l} \mbox{ for } l \in \N,
\end{eqnarray*}
where $k \in \N$ and $\xi_k$'s are the standard orthonormal basis of $\ell^2\N$. 
Define a self-adjoint operator $H$ on $\ell^2\N$ by 
\begin{equation*}
H \xi_k = (\log q_K^k) \xi_k. 
\end{equation*}
Then we can check 
\begin{equation*}
e^{itH} \pi_w(fv_k) e^{-itH} = q_K^{itk} \pi_w(fv_k) 
\end{equation*}
for $f \in C(Y_K)$ and $k \in \N$, and hence $e^{itH} \pi_w(a) e^{-itH} = \pi_w(\sigma_t(a))$ for any $a \in A_K$. 
In addition, for $f \in C(Y_K)$
\begin{eqnarray*}
\frac{\mathrm{Tr}(e^{-\beta H}\pi_w(f))}{\mathrm{Tr}(e^{-\beta H})}
= (1-q_K^{-\beta})\sum_{k=0}^{\infty} q_K^{-k\beta}f(kw) = \varphi_{\beta,w}(f). 
\end{eqnarray*}
Hence $\varphi_{\beta, w}$ is of type I. 
Moreover, the partition function is $\mathrm{Tr}(e^{-\beta H}) = (1-q_K^{-\beta})^{-1}$. 

\end{proof}

The affine isomorphisms given in Lemma \ref{lem1}--\ref{lem3} are all $\GKab$-equivariant. 
This implies that the action of $\GKab$ on the set of all extremal KMS$_{\beta}$-states is free and transitive. 
Hence Proposition \ref{KMS} implies conditions (1), (2) in Definition \ref{loc}. 

\subsection{The Arithmetic Subalgebra}

We show the existence of the arithmetic subalgebra. The definition of the arithmetic subalgebra is contained in Definition \ref{loc}. 

\begin{prop} \label{arithm}
For any local field $K$ of characteristic zero, its Bost-Connes system $\mathcal{A}_K$ possesses an arithmetic subalgebra. 
\end{prop}

The proof follows the line of Yalkinoglu (\cite{Y}), using Grothendieck-Galois correspondence. 
Although the proof could be much simpler, this proof illustrates the structure of the arithmetic subalgebra. 

For $n,m\geq 1$, let 
\begin{equation*}
C_{n,m} = \langle \pi^n \rangle \times U_K^{(m)} = \langle \pi^n \rangle \times (1+\p^m). 
\end{equation*}
By the existence theorem of the class field theory, there exists a finite abelian extension $K_{n,m}$ of $K$ such that 
$K^*/C_{n,m}$ is isomorphic to $G(K_{n,m}/K)$ by the reciprocity map. 
Here and henceforth, we identify $K^*/C_{n,m}$ and $G(K_{n,m}/K)$ via the reciprocity map, and they are denoted by $G_{n,m}$. 

Let 
\begin{equation*}
Y_{n,m} = \O_K/\p^m \times_{\O_K^*/U_K^{(m)}} G_{n,m}. 
\end{equation*}
Note that $\O_K^*/U_K^{(m)} = (\O_K/\p^m)^*$. 

\begin{lem} \label{lema}
We have $\displaystyle Y_K = \lim_{\longleftarrow } Y_{n,m}$. 
\end{lem}
\begin{proof}
It is enough to show that the natural map $\displaystyle Y_K \rightarrow \lim_{\longleftarrow } Y_{n,m}$ is injective. 
For $n,m\geq 1$, we denote the image of $[a,\alpha] \in Y_K$ on $Y_{n,m}$ by $[a_m,\alpha_{n,m}]$. 
Let $[a,\alpha],[b,\beta] \in Y_K$ and assume $[a_m,\alpha_{n,m}] = [b_m,\beta_{n,m}]$ for every $n,m\geq 1$. 
Then there exists $s_m \in \O_K^*/U_K^{(m)}$ such that $(a_m,\alpha_{n,m}) = (b_ms_m,[s_m]_K^{-1}\beta_{n,m})$. 
The element $s_m$ is unique for each $m$ since $\O_K^*/U_K^{(m)}\rightarrow G_{n,m}$ is injective, and 
$s_m$ does not depend on $n$ because the following diagram is commutative :
\begin{eqnarray*}
\xymatrix{
 & G_{n,m} \ar[dd] \\
\O_K^*/U_K^{(m)} \ar[ru] \ar[rd] & \\
 & G_{1,m}. }
\end{eqnarray*}
Taking the limit of $s_m$, we have a unique element $\displaystyle s \in \O_K^* = \lim_{\longleftarrow} \O_K/U_K^{(m)}$ such that 
$[a,\alpha] = [bs,[s]_K^{-1}\beta]$. 
\end{proof}

Note that $Y_{n,m}$ is a finite $G_{n,m}$-set. 
We would like to apply the Grothendieck-Galois correspondence theorem to $Y_{n,m}$. 
First, we decompose $Y_{n,m}$ into $G_{n,m}$-orbits. 

\begin{lem} \label{lemb}
We have $Y_{n,m} = \bigsqcup_{k=0}^{m} G_{n,m-k}$. 
Here, $G_{n,m-k}$ is identified with $\pi^k (\O_K/\p^m)^* \times_{(\O_K/\p^m)^*} G_{n,m}$ via 
$\alpha \mapsto [\pi^k,{\overline \alpha}]$, where ${\overline \alpha}$ is a lift of $\alpha$ with respect to 
the quotient map $G_{n,m} \rightarrow G_{n,m-k}$. 
\end{lem}
\begin{proof}
Note that $G_{n,m} = (\langle \pi \rangle / \langle \pi^n \rangle) \times ( \O_K^* / U_K^{(m)} )$. 
We have 
\begin{eqnarray*}
\pi^k (\O_K/\p^m)^* \times_{(\O_K/\p^m)^*} G_{n,m} \cong G_{n,m} / U_K^{(m-k)} = G_{n,m-k}. 
\end{eqnarray*}
\end{proof}

\begin{proof}[Proof of Proposition \ref{arithm}]. 
Let $E_K$ be the algebra of $\Kab$-valued locally constant functions on $Y_K$ 
which is $\GKab$-equivariant. 
We show that $E_K \rtimes \N$ is an arithmetic subalgebra. 
In order to show this, we give an inductive limit structure on $E_K$. 
By the Grothendieck-Galois correspondence, There exists a finite \'etale $K$-algebra $E_{n,m}$ such that 
$\mathrm{Hom}(E_{n,m},\Kab)$ is naturally isomorphic to $Y_{n,m}$ as $\GKab$-sets (the action of $\GKab$ factors through $G_{n,m}$). 
Namely, $E_{n,m}$ is given as the algebra of $\GKab$-equivariant $\Kab$-valued functions on $Y_{n,m}$. 
Since $\displaystyle Y_K = \lim_{\longleftarrow } Y_{n,m}$ by Lemma \ref{lema}, we have $\displaystyle E_K = \lim_{\longrightarrow} E_{n,m}$. 

By Lemma \ref{lemb}, we have 
\begin{equation*}
E_{n,m} = \prod_{k=0}^m K_{n,m-k}. 
\end{equation*}
By Proposition \ref{KMS}, extremal KMS$_{\infty}$-states are the evaluation maps at the points of $Y_K^* = \O_K^* \times_{\O_K^*} \GKab$. 
For $w \in Y_K^*$, the image of $w$ on $Y_{n,m}$ is denoted by $w_{n,m}$. 
Then we have the following commutative diagram : 
\begin{eqnarray*}
\xymatrix{ 
E_{n,m} \ar[r] \ar[d]_{\mathrm{ev}_{w_{n,m}}} & E_K \ar[d]_{\mathrm{ev}_w} \\
K_{n,m} \ar[r] & \Kab , }
\end{eqnarray*}
and the left vertical map is the composition of the projection $E_{n,m}\rightarrow K_{n,m}$ and the automorphism coresponding to $w_{n,m} \in G(K_{n,m}/K)$. 
Hence the left vertical map is surjective, and $\varphi_{\infty,w}(E_K) = \bigcup_{n,m} K_{n,m} = \Kab$. 
This shows the axiom (b) in Definition \ref{loc} holds, and the axiom (c) is automatic by $\GKab$-equivariance of functions of $E_{n,m}$. 

Next, we show that $E_K \otimes_K \C$ coincides with the algebra of all locally constant functions on $Y_K$, 
which implies the density of $(E_K \rtimes \N ) \otimes_K \C$ in $A_K$. 
It is enough to show that $E_{n,m} \otimes_K \C = C(Y_{n,m})$. 
Clearly $E_{n,m} \otimes_K \C$ is contained in $C(Y_K)$, and we can calculate the dimension of $E_{n,m}$ as 
\begin{eqnarray*}
\dim E_{n,m} = \sum_{k=0}^m [K_{n,m-k} : K] = \sum_{k=0}^m |G_{n,m-k}| = |Y_{n,m}|. 
\end{eqnarray*}
Comparing the dimension, we have $E_{n,m} \otimes_K \C = C(Y_{n,m})$. 
Therefore, the axiom (a) in Definition \ref{loc} also holds. Hence $E_K \rtimes \N$ is an arithmetic subalgebra. 
\end{proof}

By Proposition \ref{KMS} and \ref{arithm}, the proof of Theorem \ref{main} is completed. 

\section{Connection with Bost-Connes system for number fields} \label{connection}
By the work of  Laca-Neshveyev-Trivkovi\'c \cite{LNT}, 
for any number field $K$ and any finite extension $L$ of $K$,  
there is a natural $\R$-equivariant $C^*$-correspondence between their Bost-Connes systems. 
Similarly, if we consider a localization of a number field with respect to a finite prime, 
there is a natural $C^*$-correspondence from Bost-Connes system for number fields to the local one. 

Let $K$ be a number field and $\p$ be a finite prime of $K$. 
The Bost-Connes systems for $K$ and $K_p$ are denoted by by $A_K$ and $A_{\p}$. 
That is, 
\begin{equation*}
A_K = C(Y_K) \rtimes I_K\mbox{ and }A_{\p} = C(Y_{\p}) \rtimes J_K, 
\end{equation*}
where 
$Y_K = \hat{\O}_K \times_{\hat{\O}^*_K} G(\Kab/K)$ and $I_K$ is the ideal semigroup of $K$. 
Let 
\begin{eqnarray*}
&&X_K = \Af \times_{\hat{\O}^*_K} G(\Kab/K),\ X_{\p} = K_{\p} \times_{\hat{\O}^*_{\p}} G(\Kab_{\p}/K_{\p}),\\ 
&&\tilde{A}_K = C_0(X_K) \rtimes J_K,\ \tilde{A}_{\p} = C_0(X_{\p}) \rtimes \Z. 
\end{eqnarray*}
Then $A_K$ and $A_{\p}$ is full-corners of $\tilde{A}_K$ and $\tilde{A}_{\p}$ respectively. 
In addition, we have a natural inclusion $X_{\p} \subset X_K$ defined by $[\rho,\alpha] \mapsto [(\rho,1),\alpha|_K]$. 
This induces the natural inclusion $Y_{\p} \subset Y_K$. 
The corresponding quotient map $C_0(X_K) \rightarrow C_0(X_{\p})$ is denoted by $\pi$. 

Define the (right) Hilbert $\tilde{A}_{\p}$-module $\tilde{E}_{\p}$ by 
\begin{equation*}
\tilde{E}_{\p} = C^*(J_K) \otimes_{C^*(\Z)} C_0(X_{\p}) \rtimes \Z, 
\end{equation*}
where $\Z$ is identified as the subgroup of $J_K$ generated by $\p$, and the action of $C^*(\Z)$ on $C_0(X_{\p}) \rtimes \Z$ is the natural one. 
Note that $\tilde{E}_{\p}$ is a full Hilbert $\tilde{A}_{\p}$-module. 
We can define the action of $\tilde{A}_K$ on $\tilde{E}_{\p}$ by
\begin{equation*}
(fu_{\a})(u_{\b} \otimes g) = u_{\a \b} \otimes \pi((\a \b)^{-1}.f)g, 
\end{equation*}
where $f \in C_0(X_K)$, $\a,\b \in J_K$ and $g \in C_0(X_{\p})$. Hence $\tilde{E}_{\p}$ defines an $\tilde{A}_K$-$\tilde{A}_{\p}$ correspondence. 

Finally, let $E_{\p} = 1_{Y_K} \tilde{E}_{\p} 1_{Y_{\p}}$. 
Then $E_{\p}$ is a full Hilbert $A_{\p}$-module, and by restricting the action of $\tilde{A}_K$ on $A_K$, $E_{\p}$ defines an $A_K$-$A_{\p}$ correspondence. 

$E_{\p}$ is actually an $\R$-equivariant correspondence. 
Define a one-parameter group of isometries $\tilde{U}_t$ on $\tilde{E}_{\p}$ by 
\begin{equation*}
U_t(v_{\a} \otimes f) = N(\a)^{it} v_{\a} \otimes f,  
\end{equation*}
where $t \in \R, \a \in J_K, f \in C_0(X_{\p})$ and $N$ is the ideal norm. 
Let $U_t$ be the one-parameter group of isometries on $E_{\p}$ obtained by restricting $\tilde{U}_t$ on $E_{\p}$. 
Then we have 
\begin{equation*}
\langle U_t\xi, U_t\eta \rangle = \sigma_t (\langle \xi, \eta \rangle),\ U_ta\xi = \sigma_t(a)U_t\xi,
\end{equation*}
where $\xi,\eta \in E_{\p}$ and $a \in A_K$. 
Hence $(E_{\p}),U_t$ is an $\R$-equivariant correspondence. 

By using $C^*$-correspondence, we can consider the {\it induction of KMS-states} (cf. \cite{LNT}, \cite{LN}). 
If $(E,U)$ is an $\R$-euivariant $(A,\sigma_t^A)$-$(B,\sigma_t^B)$-correspondence between two $C^*$-dynamical systems and 
if $\varphi$ is a KMS$_{\beta}$-weight of $(B,\sigma_t^B)$, $\Ind_E^U \varphi$ denotes the composition of the homomorphism 
$A \rightarrow \BB(E)$ and the naturally induced KMS$_{\beta}$-weight of $\BB(E)$. 
In the case of extensions of number fields, the induction of KMS-states behaves compatibly with the natural inclusion 
$K \hookrightarrow L$ (although the temperature changes --- 
we have to modify the time evolution in order to make $C^*$-correspondence $\R$-equivariant between extensions). 
A similar phenomenon happens in the case of localizations and the method is the same as \cite{LNT}.  

Although the form of $E_{\p}$ in the above is better to comprehend, we need another description of the $C^*$-correspondence. 
Let 
\begin{equation*}
Z_{\p} = J_K \times_{\Z} X_{\p},
\end{equation*}
where $\Z$ is identified with $\p^\Z$ as usual. 
Then there are two natural maps $X_{\p} \hookrightarrow Z_{\p}$ which embeds $X_{\p}$ to the image of $\{1\} \times X_{\p}$
 and $Z_{\p} \rightarrow X_K$ which sends $[\a,z]$ to $\a z$. Both maps are equivariant for appropriate groups, 
 and we can see that $1_{X_{\p}} C_0(Z_{\p}) \rtimes J_K 1_{X_{\p}}$ is naturally isomorphic to $\tilde{A}_{\p}$, 
so $C_0(Z_{\p}) \rtimes J_K 1_{X_{\p}}$ is an $\tilde{A}_K$-$\tilde{A}_{\p}$-correspondence. This is in fact an $\R$-equivariant correspondence if we define
 the $\R$-action $U$ on $C_0(Z_{\p}) \rtimes J_K 1_{X_{\p}}$ in the usual way. 

\begin{lem}\rm{(cf. \cite[Lemma 4.2]{LNT})} 
The $\tilde{A}_K$-$\tilde{A}_{\p}$-correspondence $C_0(Z_{\p}) \rtimes J_K 1_{X_{\p}}$ is $\R$-equivariantly isomorphic to $\tilde{E}_{\p}$. 
\end{lem}
\begin{proof}
Define a map $C^*(J_K) \otimes_{C^*(\Z)} C_0(X_{\p}) \rtimes \Z \rightarrow C_0(Z_{\p}) \rtimes J_K 1_{X_{\p}}$ by 
\begin{equation*}
u_{\a} \otimes fu_{\p^n} \mapsto (\a .f) u_{\a \p^n}
\end{equation*}
for $\a \in J_K$, $n \in \Z$ and $f \in C_0(C_{\p})$. This is well defined because we can see that $(\a .f) u_{\a \p^n} 1_{X_{\p}} = (\a .f) u_{\a \p^n}$. 
For the converse map, note that 
$fu_{\a} 1_{X_{\p}}$ is nonzero for $f \in C_0(Z_{\p})$ and $\a \in J_K$ only if $\mathrm{supp} f$ is contained in $\a X_{\p}$. 
So we can define the map $C_0(Z_{\p}) \rtimes J_K 1_{X_{\p}} \rightarrow C^*(J_K) \otimes_{C^*(\Z)} C_0(X_{\p}) \rtimes \Z$ by 
\begin{equation*}
fu_{\a} 1_{X_{\p}} \mapsto u_{\a} \otimes \a^{-1}.f
\end{equation*}
for $f \in C_0(Z_{\p})$ and $\a \in J_K$ such that $fu_{\a} 1_{X_{\p}}$ is nonzero. 
Both maps are indeed $\R$-equivariant homomorphisms for $\tilde{A}_K$-$\tilde{A}_{\p}$-correspondence and they are the converse for each other. 
\end{proof}

If $A$ is a $C^*$-algebra and $p$ is a full projection of $A$, the induction of KMS-state of $pAp$ by the correspondence $Ap$ is relatively easy to describe. 
This is the purpose of the above lemma.  

For a moment, we consider the general case. 
Let $G$ be a countable abelian group acting on a locally compact space $X$. Let $N:G\rightarrow \R_+$ be a group homomorphism. 
Let $A=C_0(X) \rtimes G$ and $\sigma_t$ be a time evolution on $A$ determined by $\sigma_t(fu_g)=N(g)^{it} fu_g$ for $f \in C_0(X)$ and $g \in G$. 
Let $Y$ be a compact open subset of $X$ satisfying $\displaystyle X= \bigcup_{g \in G} gY$ (under this assumption, $p=1_Y \in C_0(X)$ is a full projection). 
The canonical conditional expectation $A\rightarrow C_0(X)$ and its restriction onto $pAp$ are both denoted by $E$. 
The $\R$-actions which are naturally induced on the correspondence $Ap$ and $pA$ are both denoted by $U$. 
The following lemma is observed in the proof of \cite[Proposition 4.5]{LNT}. 

\begin{lem} \label{induction}
Under the above setting, the following holds:
\begin{enumerate}
\item Let $\mu$ be a measure on $Y$ and assume that $\varphi= \mu \circ E$ is a positive KMS$_{\beta}$-functional of $pAp$. 
Let $\psi = \Ind_{Ap}^U \varphi$. Then $\psi = \tilde{\mu} \circ E$, where $\tilde{\mu}$ is the measure on $X$ determined by the following formula: 
\begin{equation*}
\tilde{\mu}(Z \cap gY) = N(g)^{-\beta}\mu(g^{-1}Z \cap Y), 
\end{equation*}
for any Borel subset $Z$ of $X$ and $g \in G$. 
\item Let $\nu$ be a measure on $X$ and assume that $\psi= \nu \circ E$ is a KMS$_{\beta}$-functional of $A$. 
Let $\varphi = \Ind_{pA}^U \psi$. Then $\varphi = \nu|_Y \circ E$. 
\end{enumerate} 
\end{lem}

Note that in the above proposition, induced KMS$_{\beta}$-weights need not be functionals. 

By Lemma \ref{lem2}, there is a natural one-to-one correspondence between 
positive KMS$_{\beta}$-functionals of $A_{\p}$ and finite measures on $Y_{\p}^*$. 
Let $p$ be the rational prime which is below $\p$ and let $f$ be the inertia degree. 
A finite measure $\nu$ on $Y_{\p}^*$ corresponds to the KMS$_{\beta}$-functional $\varphi_{\beta,\nu} = \mu \circ E$, 
where $\mu$ is the measure on $X_{\p}$ satisfying $\mu|_{Y_{\p}} = \nu$ and the scaling formula 
$\mu(kZ) = p^{-f\beta k}\mu(Z)$ for any $k \in \Z$ and any Borel subset $Z\subset X_{\p}$. 
Note that in this correspondence we deleted the normalization factor $1-p^{-f\beta}$, 
so this notation is not compatible with the notation in Proposition \ref{KMS}. 
If $\varphi_{\beta, \nu}$ is a KMS$_{\beta}$-state, then $\nu$ satisfies $\nu(Y_{\p}^*) = 1-p^{-f\beta}$. 

Since we have a similar correspondence in the case of number fields (cf.\cite{LLN}), we use the same notation for KMS-states of $A_K$. 

The following proposition says that the induction of KMS-states via $E_{\p}$ behaves compatibly 
with the natural homomorphism $G(\Kab_{\p}/K_{\p})\rightarrow G(\Kab/K)$: 

\begin{prop}
Let $i_{\p}: Y_{\p}^* \rightarrow Y_{K}^*$ be the inclusion map. 
Let $\varphi=\varphi_{\beta,\nu}$ be a KMS$_{\beta}$-state of $A_{\p}$, and let $\Phi = \Ind_{E_{\p}}^U \varphi$. 
Then we have the following: 
\begin{enumerate}
\item If $\beta > 1$, then $\Phi = \varphi_{\beta,i_{\p *}(\nu)}$. In particular, $\Phi(1)=\zeta_K(\beta)(1-p^{-f\beta})$. 
\item If $\beta \leq 1$, then $\Phi(1) = + \infty$. 
\end{enumerate}
\end{prop}

\begin{proof}
Since $E_{\p} = 1_{Y_K}\tilde{A}_K \otimes_{\tilde{A}_K} \tilde{E}_{\p} \otimes_{\tilde{A}_{\p}}\tilde{A}_{\p}1_{Y_{\p}}$, 
we can know $\Phi=\Ind_{E_{\p}}^U \varphi$ by using Lemma \ref{induction} repeatedly. 
We have $\Phi= \varphi_{\beta,\tilde{\mu}}$ for $\tilde{\mu} = \rho_*(\lambda)|_{Y_K}$, where $\rho:Z_{\p}\rightarrow X_K$ is the natural map and 
$\lambda$ is the measure on $Z_{\p}$ satisfying $\lambda|_{Y_{\p}^*}=\nu$ and $\lambda(\a Z) = N_K(\a)^{-\beta}\lambda(Z)$ 
for any $\a \in J_K$ and any Borel subset $Z \subset Z_{\p}$. Since $\lambda$ is concentrated on $J_KY_{\p}^*$, we have 
\begin{eqnarray*}
\rho_*(\lambda)(Y_K) &=& \sum_{\a \in J_K} \lambda (\rho^{-1}(Y_K) \cap \a Y_{\p}^*) \\
&=& \sum_{\a \in I_K} \lambda (\a Y_{\p}^*) \\
&=& \sum_{\a \in I_K} N_K(\a)^{-\beta} \lambda (Y_{\p}^*) \\
&=& \left\{ \begin{array}{ll} + \infty & \mbox{if } \beta \leq 1 \\
                             \zeta_K(\beta)(1-p^{-f\beta}) & \mbox{if } \beta >1. \end{array} \right.
\end{eqnarray*}
If $Z \subset Y_K^*$ be a Borel set, then by a similar computation, we have
\begin{eqnarray*}
\rho_*(\lambda)(Z) = \nu (\rho^{-1}(Z) \cap Y_{\p}^*) = \nu (i_{\p}^{-1}(Z)), 
\end{eqnarray*}
which concludes the proof. 
\end{proof}

\section{$C^*$-algebraic features of local Bost-Connes systems} \label{feature}

\subsection{K-theory}

We can compute $K$-theory of $A_K$. 
As we have seen in Proposition \ref{identical}, the dynamical system $(Y_K^{\natural}, \N)$ is conjugate to $(\N \times \hat{\Z} \times \O_K^*, \N)$. 
The action of $\N$ on $\N \times \hat{\Z} \times \O_K^*$ is $(k,x,a) \mapsto (k+1,x-1,a)$. 
In fact, this is conjugate to the action of $\N$ defined by $(k,x,a) \mapsto (k+1,x,a)$ via the homeomorphism defined by 
\begin{equation*}
\N \times \hat{\Z} \times \O_K^* \rightarrow \N \times \hat{\Z} \times \O_K^*,\ (k,x,a) \mapsto (k,x+k,a). 
\end{equation*}
So the crossed product $P_K=C_0(Y_K^{\natural}) \rtimes \N$ is isomorphic to $\KK \otimes C(\hat{\Z} \times \O_K^*)$. 
The closed subspace $Y_K^0$ is identified with $\hat{\Z}$, and the quotient $A_K/P_K$ is isomorphic to $C(Y_K^0) \rtimes \Z = C(\hat{\Z}) \rtimes \Z$. 
So we have the following exact sequence: 
\begin{equation*}
\xymatrix{
0 \ar[r] & \KK \otimes C(Y_K^*) \ar[r] & A_K \ar[r] & C(\hat{\Z}) \rtimes \Z \ar[r] & 0. 
}
\end{equation*}

Let $B = C(\hat{\Z}) \rtimes \Z$. 
It is known that $K_0(B) = \Q$ and $K_1(B) = \Z$. 
The generator of $K_1$ is the unitary $u_B$ corresponding to $1 \in \Z$, and 
the isomorphism $K_0(B)\rightarrow \Q$ is induced by the unique tracial state of $B$. 

\begin{prop} \label{ktheory}
Let $K$ be a local field of characteristic zero. 
Then $K_1(A_K) = 0$, and there exists an exact sequence 
\begin{equation*}
0 \rightarrow C(Y_K^*,\Z)/\Z1 \rightarrow K_0(A_K) \rightarrow \Q \rightarrow 0,
\end{equation*}
where $1$ is the constant function on $Y_K^*$. 
\end{prop}

\begin{proof}
We can easily see $K_1(\tilde{A}_K) = 0$ by using the Pimsner-Voiculescu exact sequence. 
So we have $K_1(A_K)=0$ because $A_K$ is Morita equivalent to $\tilde{A}_K$. 
By using six-term exact sequence, we have
\begin{equation*}
\xymatrix{
0 \ar[r] & K_1(B) \ar[r]^{\delta} & K_0(P_K) \ar[r] & K_0(A_K) \ar[r] & K_0(B) \ar[r] & 0,
} 
\end{equation*}
where $\delta$ is the index map. 
By the above observations, $K_0(P_K)$ is identified with $C(Y_K^*,\Z)$. 
Let $v \in A_K$ be the isometry corresponding to $1 \in \N$. Then the image of $v$ on $B$ coincides with $u_B$. 
Hence we have 
\begin{equation*}
\delta([u_B]_1) = [1-v^*v]_0 - [1-vv^*]_0 = -[1_{Y_K^*}]_0. 
\end{equation*}
The element $-[1_{Y_K^*}]_0$ of $K_0(P_K)$ is identified with the constant function $-1 \in C(Y_K^*,\Z)$. 
This implies that $\delta(K_1(B)) = \Z 1$. 
\end{proof}

Note that the group $C(Y_K^*,\Z)/\Z1$ is abstractly isomorphic to the free abelian group of infinite rank. 

\subsection{The primitive ideal space}

The primitive ideal space can be determined in the same way as in \cite{T}, \cite{LR} by applying Williams's theorem \cite{W} to $\tilde{A}_K$. 
Since the action of $\Z$ on $X_K$ is free, any primitive ideal is of the form of $I_x = \ker \pi_x$ for $x \in Y_K$, where 
$\pi_x$ is the representation of $A_K$ on $\BB(\ell^2 \N)$ determined by 
\begin{equation*}
\pi_x(f) \xi_k = f(kx),\ \pi_x(v_l)\xi_k = \xi_{k+l}
\end{equation*}
for $f \in Y_K$ and $k,l \in \N$. For $x,y \in Y_K$, we have $I_x = I_y$ if and only if $\overline{\Z x} = \overline{\Z y}$ in $X_K$. 
If $x \in Y_K^0$, then $\overline{\Z x} = \overline{\Z y}$ holds if and only if $y \in Y_K^0$. 
If $x \in Y_K^{\natural}$, then there exists a unique element $y \in Y_K^*$ satisfying $\overline{\Z x} = \overline{\Z y}$. 
So by fixing an element of $Y_K^0$, we can identify $\Prim A_K$ with $Y_K^* \cup \{x_0\}$. 
We can see that the ideal $I_{x_0}$ is identical to $P_K = C_0(Y_K^{\natural}) \rtimes \N$. 
So we have the following proposition: 
\begin{prop}
We have 
\begin{equation*}
\Prim A_K = \{ \ker \pi_w\ |\ w \in \GKab \} \cup \{P_K\},
\end{equation*}
where $\pi_w$ for $w \in \GKab$ is the representation that appeared in the proof of Proposition \ref{KMS}. 
\end{prop}

It is interesting to compare it with the case of number fields (cf. \cite{T}). 
In the case of number fields, the irreducible representations which arise from the Galois group are all faithful, 
so they give the same point in the primitive ideal space (although they are not unitarily equivalent). 
Contrary to that, such representations for local fields actually give different points in $\Prim A_K$ and the Galois group is embedded in $\Prim A_K$. 
This is mainly because $Y_K^*$ is clopen if $K$ is a local field of characteristic zero and is not open if $K$ is a number field. 
Also, for number fields the space which amounts to the point $P_K$ in the local case is not one point --- it is an infinite dimensional torus. 
This difference is caused by the difference of actions. In the case of number fields, the action of the ideal group on the base space is not free and 
the dual of the isotropy group on the zero part appears in the primitive ideal space. 
For local fields, as we remarked before, the action of $\N$ on $Y_K$ is free. 

At last, we would like to describe the topology of $\Prim A_K$. 
The topology is determined by using the fact that the quotient map $Y_K \rightarrow \Prim A_K$ is an open map. 
Hence the inclusion map $Y_K^* \hookrightarrow \Prim A_K$ is a homeomorphism onto its range. 
The subset $Y_K^* \subset \Prim A_K$ is open and compact, but not closed. 
The only open set which contains $P_K$ is the entire set. 
In order to see this, note that for any $x \in Y_K^*$, the sequence $\{ kx \}_{k=1}^{\infty}$ has a subsequence converging to a point $y$ in $Y_K^0$. 
Let $U$ be an open set which intersects with $Y_K^0$. 
Then for some $k \in \Z$ we have $kU \ni y$ and hence $kU$ contains a point which is equivalent to $x$. 
This implies that the image of $U$ is the entire set.

\section*{Acknowledgments}
This work was supported by the Program for Leading Graduate Schools, MEXT, Japan and JSPS KAKENHI Grant Number 13J01197. 
The author would like to thank to Yasuyuki Kawahigashi for fruitful conversations. 

\bibliographystyle{amsplain} 

\bibliography{ref}

\end{document}